\documentclass[10pt]{amsart}
\usepackage{amssymb, amsmath}
\usepackage{graphics}
\usepackage{latexsym}
\usepackage{amsmath}
\usepackage{amssymb,amsthm,amsfonts}
\usepackage{amscd}
\usepackage[arrow, matrix, curve]{xy}
\usepackage{syntonly}
\usepackage{yfonts}
\usepackage{tikz-cd}
\usepackage{filecontents}
\usepackage{mathtools}
\usepackage{stmaryrd}
\usepackage{MnSymbol}
\DeclareMathOperator{\spec}{Spec}
\DeclareMathOperator{\gal}{Gal}
\DeclareMathOperator{\aut}{Aut}
\DeclareMathOperator{\ord}{ord}
\DeclareMathOperator{\gl}{GL}
\DeclareMathOperator{\pr}{pr}
\DeclareMathOperator{\h}{H}
\DeclareMathOperator{\reg}{Reg}
\theoremstyle{plain}
\newtheorem{thm}{Theorem}[section]
\newtheorem{theorem}[thm]{Theorem}
\newtheorem{lemma}[thm]{Lemma}
\newtheorem{corollary}[thm]{Corollary}
\newtheorem{proposition}[thm]{Proposition}

\theoremstyle{definition}

\newtheorem{definition}[thm]{Definition}

\newtheorem{example}[thm]{Example}

\numberwithin{equation}{thm}

\newcommand{\sF}{{\mathcal F}}

\newcommand{\sL}{{\mathcal L}}
\newcommand{\sM}{{\mathcal M}}
\newcommand{\sN}{{\mathcal N}}
\newcommand{\sO}{{\mathcal O}}

\newcommand{\C}{{\mathbb C}}

\newcommand{\N}{{\mathbb N}}
\renewcommand{\P}{{\mathbb P}}

\newcommand{\Z}{{\mathbb Z}}

\newcommand{\codim}{{\rm codim}}

\begin{document}
\title{on metabelian galois coverings of complex algebraic varieties}
\author{Abolfazl Mohajer}
\address{Universit\"{a}t Mainz, Fachbereich 08, Institut f\"ur Mathematik, 55099 Mainz, Germany}
\email{mohajer@uni-mainz.de}
\subjclass{14A10, 14A15, 14E20, 14E22}
\keywords{Algebraic variety, Galois covering, ramified covering}
\maketitle
\begin{abstract}
Let $f:X\to Y$ be a finite ramified Galois covering of algebraic varieties defined over the complex numbers. In this paper, we prove some structure theorems for such coverings in the case that the non-abelian Galois group of the cover is metabelian (in particular metacyclic). Our results extend the previous ones, obtained by several authors, in the case of abelian and dihedral Galois coverings. In analogy with the abelian (and generalizing the dihedral) case, we find "building data" of metabelian covers from which the cover can be reconstructed. In addition to analyzing several examples, we compute invariants and eigensheaves of the group actions of the coverings under study with some applications.
\end{abstract}	
\section{Introduction}
The theory of branched (or ramified) coverings has its origins in continuation of analytic functions and the attempts to find maximal analytic continuations of a given function. Taking complex root functions such as $f(z)=\sqrt{z}$ shows that they are multi-valued in certain subdomains of the complex plane, so when trying to continue along the closed curve one might arrive at another branch of the multivalued function, not the original one, so the values do not match. The idea of Riemann surfaces offers a geometric way to deal with this problem by introducing a Riemann surface as the natural domain of $\sqrt{z}$ and similar functions. One can picture the Riemann surface for $\sqrt{z}$ as two sheets coming together at $z=0$. As such, it can be considered as a covering $\P^1\to \P^1$ of degree 2 branched over $0$ and $\infty$. \par 
The Riemann surfaces appeared in Bernhard Riemann's inaugural thesis in 1851 and are considered the first examples of covering spaces. Riemann's breakthrough, the \emph{Riemann existence theorem} asserts that such coverings of (compact) Riemann surfaces are classified by the so-called (monodromy) permutation representation of the complement of the branch points. \par
The study of Galois coverings of algebraic varieties, i.e. maps $\pi:X\to Y$ such that $\pi$ exhibits $Y$ as a quotient $X/G$ for a finite group $G$ has since then been an interesting and fruitful subject with many applications in algebraic geometry. This is equaivalent to solving equations with the coefficients in the function filed of the base variety $Y$. As some applications, one can point out the construction of algebraic surfaces with a  prescribed  Chern  invariants, see \cite{Per81}, or the use of abelian covers to inquire the existence of Shimura subvarieties in the moduli space of abelian varieties, see \cite{MZ}, \cite{M}. However, in most cases the Galois coverings were abelian, due to rich and well-known properties, such as simple representation theory or more basic equations, which made them as, as described in (\cite{Tok02}), "quite user-friendly". \par 
According to \cite{CP}, Comessatti \cite{Com30} was the first to study Galois covers with abelian Galois group $G$ and their topological properties. In \cite{P1}, Pardini described the abelian Galois covers of non-singular algebraic varieties over the complex numbers (although her methods can be extended to any algebraically closed field of characteristic coprime to $n$). The analysis in this case uses the decomposition of $\pi_*\sO_X=\bigoplus\limits_{\chi\in G^{*}}\sL_{\chi}$ and the $\sO_X[G]$-sheaves $\sL_{\chi}$ are invertible. The cover is then determined by the so-called \emph{building data}, a collection of data $\{\sL_{\rho}\otimes \sL_{\chi^{\prime}}\to \sL_{\chi\chi^{\prime}} \}_{\chi, \chi^{\prime}}$ (or equivalently a collection of global sections $\{\sL_{\chi}^{-1}\otimes \sL_{\chi^{\prime}}^{-1}\otimes \sL_{\chi\chi^{\prime}} \}_{\chi, \chi^{\prime}}$) which endows $\bigoplus\limits_{\chi\in G^{*}}\sL_{\chi}$  with the structure of an $\sO_Y$-algebra constructing the cover as $\pi:\spec(\bigoplus\limits_{\chi\in G^{*}}\sL_{\chi})\to Y$. \par
The above mentioned results motivate the study of non-abelian Galois covers using such $\sO_X[G]$-sheaves to construct building data for the cover. The non-abelian Galois covers are, at least in comparison with the abelian case, relatively unexplored. The studies are mostly sporadic and dedicated to elementary and small non-abelian groups, see for example \cite{Tok02} or \cite{Ea11}. More recently, F. Catanese and F. Perroni described the algebro-geometric properties of the dihedral Galois covers of algebraic varieties in \cite{CP} using the theory of cyclic covers from \cite{C10} to construct building data for the dihedral covers, see also \cite{CLP}.\par
One of the most useful features of the dihedral group of which the analysis in \cite{CP} takes advantage is that $D_n$ is a \emph{metacyclic group}. Namely, $D_n$ sits in the following short exact sequence of groups 
\[0\to\Z/n\Z\to D_n\to \Z/2\Z\to 0\]
The two maps above give a corresponding factorization $\pi \colon X\xrightarrow{p}Z\xrightarrow{q} Y$, where $p,q$ are cyclic covers and hence the theory in
\cite{C10} or \cite{P1} can be applied. This motivates considering the metacyclic or more generally \emph{metabelian} covers (see Definition \ref{defmetabelian}) and trying to generalize the results of \cite{CP} by technics which are available for abelian covers. This is the aim of the present paper. Using the complex representation theory of metabelian and metacyclic groups, we also prove some results about the eigenspaces of the action of the group on direct image sheaf $\pi_*\mathcal{O}_X$ and the direct summands of this action. Furthermore, we describe the function fields of metacyclic covers. We also explore several examples of metabelian covers with their properties.
	
\section{Structure of abelian Galois coverings} \label{abeliancovers}
In this section  we describe the building data of an abelian cover and the construction of the cover using these data and the relations among them. Furthermore, we prove some extensions of the results of \cite{CP} and \cite{P1} to the normal varieties using theory of reflexive sheaves. Notations come mostly from \cite{C10} and \cite{P1}.  \par
Let $W$ be a smooth complex algebraic variety and $V$ a normal one and $f\colon V\to W$ be an abelian Galois covering. By this we mean precisely that there exists a finite abelian group $G$ together with a faithful action of $G$, such that $f$ realizes $W$ as the quotient of $V$ by $G$. Such Galois coverings of algebraic varieties have been studied extensively, especially in (\cite{C10} and \cite{P1}). \par
Our assumptions imply that there is a decomposition $f_*\mathcal{O}_V=\bigoplus\limits_{\chi \in G^*} L^{-1}_{\chi}$, where each $L_{\chi}$ is an invertible sheaf on $W$ on which $G$ acts by character $\chi$. So in particular, the invariant summand $L_1$ is ismorphic to $\mathcal{O}_W$. The algebra structure on $f_*\mathcal{O}_V$ is given by the ($\mathcal{O}_W$-linear) muliplication rule $\mu_{\chi\chi^{\prime}}:L^{-1}_{\chi}\otimes L^{-1}_{\chi^{\prime}}\to L^{-1}_{\chi\chi^{\prime}}$ and compatible with the action of $G$.  Consider the ramification and branch divisors $R,D$ of $f$. Note that $R$ consists precisely of those points with non-trivial stabilizers  in $V$ under the action of $G$. Furthermore under our assumptions $f$ is flat and so $D$ is a Cartier divisor and $R$ is $\mathbb{Q}$-Cartier. For every component $T$ of $R$, the subgroup of $G$ fixing elements of $T$ pointwise is a cyclic subgroup $H_T$, its so-called \emph{inertia group}. Also, if $D_i$
is an irreducible component of the branch locus $D$, then all of the elements of $f^{-1}(D_i)$ have the same inertia group which we denote by $H_i$. The order $m_i$ of $H_i$ is the ramification order of $f$ over $D_i$ and the representation of $H_i$ obtained by taking differentials and restricting to the normal space to $D_i$ is the faithful character $\chi_i$. \par
Let $D_1,\cdots, D_r$ be the irreducible components of $D$. The choice of a primitive $m$-th root of unity $\xi$ amounts to giving a map $\{1,\cdots, r\}\to G$, the image $g_i$ of $i$ under which is the generator of the cyclic group $H_i$ that is sent to $\xi^{m/m_i}=\xi_{m_i}$ by $\chi_i$. The line bundles $L_{\chi}$ and divisors $D_i$ each labelled with an element $g_i$ as described above are called the building data of the cover. These data are to satisfy the so-called \emph{fundamental relations} and
determine the cover $f:V\to W$ up to deck automorphisms. Let us write these relations down. For $i=1,\cdots, r$ and $\chi\in G^*$, let $a^{i}_{\chi}$ be the smallest positive integer such that $\chi(g_i)=\zeta^{ma^{i}_{\chi}/m_i}$. For any two characters
$\chi, \chi^{\prime}$, $0\leq a^{i}_{\chi}+a^{i}_{\chi^{\prime}}<2m_i$, so
\[\epsilon^{i}_{\chi,\chi^{\prime}}=[\frac{a^{i}_{\chi}+a^{i}_{\chi^{\prime}}}{m_i}]=
\begin{cases}
1& a^{i}_{\chi}+a^{i}_{\chi^{\prime}}\geq m_i\\
0, & a^{i}_{\chi}+a^{i}_{\chi^{\prime}}< m_i
\end{cases}\]
and we set $D_{\chi,\chi^{\prime}}=\sum\limits_{i=1}^{r} \epsilon^{i}_{\chi,\chi^{\prime}}D_i$. Then, the fundamental
relations of the cover are the following:
\begin{align} \label{fundamentalrel}
L_{\chi}+L_{\chi^{\prime}}\equiv L_{\chi\chi^{\prime}}+D_{\chi, \chi^{\prime}}.
\end{align}
In particular, if $\chi^{\prime}=\chi^{-1}$, then 
\begin{align} \label{fundamentalrelinverse}
L_{\chi}+L_{\chi^{-1}}\equiv D_{\chi,\chi^{-1}},
\end{align}
and $D_{\chi,\chi^{-1}}$ the sum of the components $D_i$, where $\chi(g_i)\neq 1$. The cover $f:V\to W$ can be recovered from the fundamental
relations \ref{fundamentalrel} by first defining the variety $X$
inside the vector bundle $\sL=\oplus_{\chi\neq 1} L_{\chi}$ by the
equations
\begin{equation} \label{fundamental eqs}
z_{\chi}z_{\chi^{\prime}}=(\prod_i s_i^{\epsilon^{i}_{\chi,\chi^{\prime}}})z_{\chi\chi^{\prime}}
\end{equation}
where $z_{\chi}$ is the fiber coordinate of the bundle $L_{\chi}$ which can also be viewed as the tautological section of pull-back of the bundle $L_{\chi}$ to $\sL$ and $s_i\in H^0(W,\sO_W(D_i))$ is the (pull-back to $\sL$ of the) defining equation for $D_i$ for $i=1,\cdots, r$. This is naturally a $W$-scheme and is flat over $W$. Conversely, for every choice of the sections $s_i$, equations \ref{fundamental eqs} define a scheme $V$ flat over $W$ which is smooth if and only if each $D_i$ is smooth, the union $\cup D_i$ has at most normal crossing singularitires and at any intersection points $\cap D_{i_l}$, the product $\prod H_{i_l}$ injects into $G$. We
therefore have the following fundamental theorem proven in \cite{P1}.
\begin{theorem} \label{abelian cover}
Let $G$ be a finite abelian group. Let $V$ be a normal algebraic variety
and $W$ a smooth one and let $f\colon V\to W$ be an abelian cover
with Galois group $G$. With the notations as above, the following set of
linear equivalences hold.
\begin{equation} \label{fundamentalrel1}
L_{\chi}+L_{\chi^{\prime}}\equiv L_{\chi\chi^{\prime}}+ D_{\chi,
\chi^{\prime}} \text{     } \forall \chi, \chi^{\prime}\in G^*.
\end{equation}
Conversely, a set of data $\{L_{\chi}\}_{\chi \in G^*}, \{D_{\chi,
\chi^{\prime}}\}$ consisting respectively of invertible sheaves and reduced effective divisors on $W$ satisfying the relation \ref{fundamentalrel1}
determines an abelian cover. When $W$ is furthermore complete,
this abelian cover is unique.
\end{theorem}
Next, let $V$ be an irreducible normal variety but $W$ not necessarily smooth.
Note that since $W=V/G$, it follows that $W$ is also an irreducible normal variety.
Consider the smooth locus $W^0$ and set $V^0\coloneqq f^{-1}(W^0)$. Then
the cover $f^0\coloneqq f\mid_{V^0}\colon V^0\to W^0$ is an abelian cover
of the type that we described above. In particular, it is determined by the line bundles $(L^0_{\chi})_{\chi \in G^*}$ such that $f^0_{*}\mathcal{O}_{V^0}=\oplus_{\chi\in G^*} L^0_{\chi}$ and reduced effective divisors $(D^0_i)$ on $W^0$ without common components such that \ref{fundamentalrel2} holds. Let us denote the natural inclusions by $i:W^0\to W, \iota:V^0\to V$. We then have the following commutative diagram
\begin{equation} \label{normal diag}
\begin{tikzcd} 
V^0 \arrow{r}{\iota} \arrow[swap]{d}{f^0} & V \arrow{d}{f} \\
W^0 \arrow{r}{i} & W
\end{tikzcd}
\end{equation}
As $\mathcal{O}_V=\iota_{*}\mathcal{O}_{V^0}$ and $\mathcal{O}_W=i_{*}\mathcal{O}_{W^0}$, one obtains $f_{*}\mathcal{O}_{V}=(f\circ \iota)_*\mathcal{O}_{V^0}=i_*f^0_{*}\mathcal{O}_{V^0}=\oplus_{\chi\in G^*}i_{*}\mathcal{O}_{W^0}(L_{\chi})$. This motivates defining $\mathcal{F}_{\chi}=i_{*}\mathcal{O}_{W^0}(L_{\chi})$. Since $W$ is a normal variety and $L_{\chi}$ is a line bundle on $W^0$, it follows that $\mathcal{F}_{\chi}$ is a reflexive sheaf on $W$ and any such sheaf is uniquely determined by its restriction to $W^0$. The multiplication $\mu:f_{*}\mathcal{O}_V\otimes_{\mathcal{O}_W}
f_{*}\mathcal{O}_V\to f_{*}\mathcal{O}_V$ gives rise to $\mu_{\chi\chi^{\prime}}:\mathcal{F}_{\chi}\otimes \mathcal{F}_{\chi{\prime}}\to \mathcal{F}_{\chi\chi^{\prime}}$ whose restriction on $W^0$ is \ref{fundamentalrel1}:
 $\mathcal{O}_{W^0}(L_{\chi}+L_{\chi{\prime}})=\mathcal{O}_{W^0}(L_{\chi\chi^{\prime}}+
D^0_{\chi, \chi^{\prime}})$. Take $D_i=\overline{D^0_i}$ to be the
closure of the divisor $D^0_i$ mentioned above. Consequently, the
multiplication map is fully determined by its restriction to
$W^0$ and the following generalization of Theorem \ref{abelian cover} holds. We remark that this Theorem is proven (somewhat implicitely) for double covers of normal varieties in \cite{CP}, p.29 (a part on p.24) using a theory of double covers developed in \cite{CP}. Here we stablish the result for all abelian covers using instead a result of \cite{H80} which states that a reflexive rank 1 sheaf on a regular scheme is invertible.
\begin{theorem} \label{singular abelian}
Let $G$ be an abelian group. Let $V$ be a normal algebraic variety and let $f\colon V\to W$ be an abelian cover with the Galois group $G$. With the notations as above, there exist a collection of rank 1 reflexive sheaves $(\mathcal{F}_{\chi})_{\chi\in G^*}$ and
divisors $(D_i)$ on $W$ such that the following set of linear equivalences hold
\begin{equation} \label{fundamentalrel2}
\mathcal{F}_{\chi}+\mathcal{F}_{\chi^{\prime}}\equiv
\mathcal{F}_{\chi\chi^{\prime}}+ D_{\chi, \chi^{\prime}} \text{
 } \forall \chi, \chi^{\prime}\in G^*.
\end{equation}
Conversely, a set of data $\{\mathcal{F}_{\chi}\}_{\chi \in G^*},
\{D_{\chi, \chi^{\prime}}\}$ as above satisfying the relation
\ref{fundamentalrel2} determines an abelian cover.
\end{theorem}
\begin{proof}
The above explanations prove the first part of the theorem. For
the converse, assume that $\{\mathcal{F}_{\chi}\}_{\chi \in G^*},
\{D_{\chi, \chi^{\prime}}\}$ are given sets of rank one reflexive sheaves and divisors subject to relations
\ref{fundamentalrel2}. Write $V=\spec(\oplus_{\chi\in G^*} \mathcal{F}_{\chi})$, where
$\oplus_{\chi\in G^*} \mathcal{F}_{\chi}$ is a sheaf of $\mathcal{O}_W$-algebras with the algebra structure given by the relations \ref{fundamentalrel2}. Let $L_{\chi}=i^*\mathcal{F}_{\chi}=\mathcal{F}_{\chi}|_{W^0}$ be the restriction of $\mathcal{F}_{\chi}$ on $W^0$. Since $\mathcal{F}_{\chi}$ is a rank one reflexive sheaf, it follows from \cite{H80}, Prop 1.9 that $L_{\chi}$ is an invertible sheaf on $W^0$. Furthermore, it holds that $\mathcal{F}_{\chi}=i_*L_{\chi}$ so that $L_{\chi}$ is uniquely determined by $\mathcal{F}_{\chi}$. By restricting the divisors $D_i$ to divisors $D^0_i$ on $W^0$, relation \ref{fundamentalrel2} restricts to relations \ref{fundamentalrel}: $L_{\chi}+L_{\chi^{\prime}}\equiv L_{\chi\chi^{\prime}}+D^0_{\chi, \chi^{\prime}}$. Now by Theorem \ref{abelian cover} these relations determine the abelian cover as claimed.
\end{proof}
\subsection{Canonical bundle formula} \label{eigenspaces}
As before, $D$ denotes the branch divisor of $f$ with irreducible components $D_i$. We have already remarked that the scheme $V$ can be constructed inside the (total space of the) vector bundle $\sL=\bigoplus\limits_{\chi\in G^*\setminus\{1\}} L^{-1}_{\chi}$ by the equations \ref{fundamental eqs} in terms of the tautological section $z_{\chi}$ of pull-back of the bundle $L_{\chi}$ to $\sL$ and the local defining equation $s_i\in H^0(W,\sO_W(D_i))$ for $D_i$. Let $p:\sL\to W$ be the bundle projection (we will use the same notation for $\sL$ and its total space). One can embed $W$ in $\sL$ by the zero section of $p$. Let the branch divisor $D$ be smooth (and reduced). As a closed subscheme, $W$ is given inside $\sL$ by the equations $z_{\chi}=0$. Let $R=f^{-1}(D)$. Suppose that the group of characters $G^*$ is generated by the characters $\chi_j, j=1,\cdots, s$. In this way, we obtain a so-called \emph{reduced building data} $L_j=L_{\chi_j}$ and any such data determines an abelian covers, see \cite{P1}. Then $V$ is defined inside $\sL$ by the equations \ref{fundamental eqs} and $f=p|_V$. Let $\sL^{\prime}=\bigoplus\limits_{j=1}^s L^{-1}_{j}$ and consider the bundle projection $p^{\prime}:\sL^{\prime}\to W$ . Note that the equations \ref{fundamental eqs} imply that $V$ has at most singularities over the singular points of the branch divisor $D$. The next result generalizes \cite{BHPV}, Lemma 17.1. 
\begin{theorem} \label{divisor cover}
Let $G$ be a finite abelian group. Let $V$ be a normal algebraic variety and $W$ a smooth one and let $f\colon V\to W$ be an abelian cover of degree $n$ with Galois group $G$ such that the linear equivalences \ref{fundamentalrel1} hold and the branch divisor $D$ is smooth. Then we have:
\begin{enumerate}
\item $\sO_V(R_{red})=f^*\sL^{\prime}$.
\item $f^*D=nR_{red}$.
\item $K_V=f^*(K_W\otimes {\sL^{\prime}}^{n-1})$. 
\end{enumerate}
\end{theorem}
\begin{proof}
Consider the ideal sheaf $\sO_\sL(B)$ of the closed subscheme $W$ in $\sL$. By the equations of $W$ given by the reduced building data mentioned above, it follows that $\sO_\sL(B)=p^*\sL^{\prime}$. By considering the equations of $V$ inside $\sL$, one sees that $V$ and $W$ intersect transversaly in $R_{red}$ inside $\sL$ and so: $\sO_V(R_{red})=\sO_\sL(B)|_V=f^*\sL^{\prime}$.
Part (3) is a result of the usual Hurwitz formula and the fact that the ramification divisor of $f$ is equal to $(n-1)R_{red}$ together with (1) above.  
\end{proof}
Now suppose that $V$ and $W$ are normal varieties, not necessarily smooth. Using Theorem \ref{singular abelian}, we can extend Theorem \ref{divisor cover} to normal varieties. Let us remark that for a normal variety $X$, $j:X^0\hookrightarrow X$ a smooth part of $X$ having complement of codimension $\geq 2$, we have $\omega_X=j_*\omega_{X^0}$, see \cite{Ri}, Corollary (8). Furthermore, as in Theorem \ref{singular abelian}, we have reflexive sheaves $L_{\chi}=i_*L_{\chi}^0$, where $i:W^0\hookrightarrow W$ is as in diagram \ref{normal diag} and $L_{\chi}^0$ is invertible on $W^0$. In particular, $\sL^{\prime}=i_*{\sL^{\prime}}^{0}=i_*(\bigoplus\limits_{j=1}^s {(L^0_{j})}^{-1})$.
\begin{proposition} \label{singular canonical}
Let $V$ and $W$ be normal varieties and $f\colon V\to W$ be an abelian cover with Galois group $G$. With the notations as above we have $K_V=f^*(K_W\otimes {\sL^{\prime}}^{n-1})$.
\end{proposition}
\begin{proof}
Let $V^{\prime}=\reg(V)\cap f^{-1}(\reg(W))=V^0\cap f^{-1}(W^0)$ and $W^{\prime}=f(V^{\prime})$. As both $V$ and $W$ are normal and $f$ is finite by assumption, it follows that the complement of $V^{\prime}$ and hence of $W^{\prime}$ is of codimension $\geq 2$. As in the proof of Theorem \ref{singular abelian}, the sheaf $\sL^{\prime}$ defined above is reflexive. The above remarks imply that $K_W=i_*K_{W^{\prime}}$ and $K_V=\iota_*K_{V^{\prime}}$, where $\iota: V^{\prime}\hookrightarrow V$ (resp. $i:W^{\prime}\hookrightarrow W$) is the natural inclusion.  The claim now follows from this together with Theorem \ref{divisor cover}(3). 
\end{proof}
\subsection{Field Extensions}
\subsubsection{Kummer Theory} \label{Kummer}
Let $K$ be a field which contains $n$ distinct $n$-th roots of unity with $n>1$. A \emph{Kummer extension} of $K$ is a (finite) Galois field extension $L/K$ whose Galois group $G$ is abelian. Kummer theory asserts that any such Galois extension comes from adjoining $n$-th roots of unity from elements of $K^{*}$. In other words, there are $a_1,\cdots, a_t\in K^*$ such that $L=K(a_1^{\frac{1}{n}},\cdots,a_t^{\frac{1}{n}})$. Moreover, if $K^{*}$ is the multiplicative group of $K$, then Kummer extensions of exponent $n$ are in bijection with subgroups of the group $K^{*}/(K^{*})^n$. 
The correspondence can be described explicitly as follows. Given a subgroup 
\[\Delta\subseteq K^{*}/(K^{*})^n\]
the corresponding extension is given by 
\[L=K(\Delta^{\frac{1}{n}}),\]
where $\Delta^{\frac{1}{n}}=\{a^{\frac{1}{n}}\mid a\in K^{*}, a\cdotp(K^{*})^n\in \Delta\}$
it also holds that
\[\gal(L/K)\cong \Delta.\]
Now let $f\colon V\to W$ be a Galois covering of algebraic varieties. It gives an extension of the corresponding function fields $\C(W)\subseteq \C(V)$. If $G$ is the Galois group of $f$, then $\gal(\C(V)/ \C(W))\cong G$. If $G=\langle\sigma_1,\cdots, \sigma_s\rangle$ is furthermore an abeian group, then $\C(V)/ \C(W)$ is a Kummer extension. Let us use the same notation for elements of $G$, elements of $\C(V)/ \C(W)$ and the corresponding automorphisms of $V$. Setting $\ord(\sigma_i)=m_i$, by the above description one has $\C(V)=\C(W)(\sqrt[m_1]{f_1},\cdots, \sqrt[m_s]{f_s})=\C(W)(v_1,\cdots, v_s)$, with $f_i\in \C(W)$ and $\sigma_j\cdotp v_i=\zeta_{m_i}^{\delta_{ij}}\cdotp v_i$, where $\delta_{ij}$ is the Kronecker delta and $\zeta_{m_i}=\zeta_{m}^{\frac{m}{m_i}}$ is a primitive $m_i$-th root of unity ($\zeta_{m}$ being a primitive $m$-th root of unity).
\section{Structure of metabelian Galois coverings}
\begin{definition} \label{defmetabelian}
A metabelian (resp. metacyclic) group $G$ is a group that has an abelian (resp. cyclic) normal subgroup $A$ such that $G/A$ is also abelian (resp. cyclic). In other words, it is an extension of an abelian (resp. cyclic) group by an abelian (resp. cyclic) group. 
\end{definition}
The above definition is equivalent to saying that metabelian groups are precisely the solvable groups of derived length at most 2. \par
Suppose 
\begin{equation} \label{metabelian extension}
0\to A\to G\to N\to 0
\end{equation} 
is the extension mentioned in Definition \ref{defmetabelian} with $A,N$ abelian. It is straightforward to see that the very definition of a metabelian group implies that $G$ is metabelian if and only if $G$ has the following presentation
\begin{equation} \label{presentation metabelian}
\langle \sigma_1,\cdots,\sigma_s, \tau_1,\cdots,\tau_l\mid
\sigma_i\sigma_j=\sigma_j\sigma_i, \tau_i\tau_j=\tau_j\tau_i,
\sigma_i^{m_i}=1, \sigma_i\tau_j=\tau_j\sigma^{r_{1ij}}_1\cdots
\sigma^{r_{sij}}_s, \tau_j^{a_j}=\sigma^{k_{1j}}_1\cdots \sigma^{k_{sj}}_s\rangle
\end{equation}
Here $A=\langle \sigma_1,\cdots,\sigma_s \rangle$ and $N=\langle\overline{\tau}_1,\cdots,\overline{\tau}_l \rangle$, and $\overline{\tau}_j$ denotes the image of $\tau_j$ in $N=G/A$.\par
Now Let $G$ be a finite metabelian group as above, $X$ a normal algebraic variety over $\C$ with $G\subset \aut(X)$ and $Y$ a smooth complex algebraic variety such that $X/G=Y$ and the cover $\pi: X\to Y$ is Galois. We are interested in the quotient $X\to X/G\coloneqq Y$ in the case that this map yields a Galois covering of algebraic varieties. The factorization \ref{metabelian extension} gives rise to a factorization $\pi \colon X\xrightarrow{p}Z\xrightarrow{q} Y$ where $p,q$ are the corresponding intermediate abelian Galois covers, i.e., $p\colon X\to Z=X/A$
is an abelian Galois covering with Galois group $A$ and $q\colon Z\to Y=Z/N$
is an abelian Galois covering with Galois group $N$. Therefore to study the Galois covering $\pi \colon X\to Y$, it is helpful to study these intermediate abelian coverings.
We will use the theory of abelian Galois coverings that we explained in Section \ref{abeliancovers} (developed in \cite{C10} and \cite{P1}) to study $\pi \colon X\to Y$ by looking at these intermediate abelian coverings.\par
Explicitely, $Z=X/A$ is a normal variety, but not in general smooth. By Theorem \ref{singular abelian}, the cover $p\colon X\to Z$ is determined by the existence of rank one reflexive sheaves $(\mathcal{F}_{\chi})_{\chi \in A^*}$, and reduced effective divisors $(D_i)$ on $Z$ without common components such that \ref{fundamentalrel2} holds. The multiplication map $\mu_{\chi\chi^{\prime}}:\mathcal{F}_{\chi}\otimes \mathcal{F}_{\chi{\prime}}\to \mathcal{F}_{\chi\chi^{\prime}}$ is fully determined by its restriction to $Z^0$. Before stating our structure theorem, let us introduce the following notation: Suppose $\chi$ is an irreducible character of the abelian group $A=\langle \sigma_1,\cdots,\sigma_s \rangle$. Let $\tau_j\in N$. Since $A$ is a normal subgroup of $G, \tau_j^{-1}\sigma_u\tau_j\in A$ for every $u=1,\cdots, s$. We define a new character $\chi^{(1)}_j$ of $A$ by $\chi^{(1)}_j(\sigma_u)=\chi(\tau_j^{-1}\sigma_u\tau_j)$ for every $u=1,\cdots, s$. Since $\chi$ is an irreducible character, $\chi^{(1)}_j$ is also irreducible. In particular for each $\gamma\in \N$ one can define a character $\chi^{(\gamma)}_{j}$ of $A$ by setting $\chi^{(\gamma)}_{j}(\sigma_u)=\chi(\tau_j^{-\gamma}\sigma_u\tau_j^{\gamma})$. By presentation \ref{presentation metabelian}, it is clear that $\chi^{(a_j)}_{j}=\chi.$ \par
Now, we are ready to state our theorem that describes the structure of metabelian Galois covers. 
\begin{theorem} \label{structure of metabelian}(Structure theorem for metabelian covers) A
metabelian Galois cover $\pi \colon X\to Y$ is determined by the following data:
\begin{enumerate}
\item Line bundles $(L_{\eta})_{\eta\in N^*}$ and reduced effective divisors $B_1\cdots, B_l$ on $Y$ such that $L_{\eta}+L_{\eta^{\prime}}\equiv \sum \epsilon^{i}_{\eta\eta^{\prime}} B_i$.
\item Reduced effective Weil divisors $D_1,\cdots, D_n$ on $Z=\spec
(\oplus L_{\eta}^{-1})$ identifying the character $\chi_i$ with $i$ such that  $\overline{\tau_j}(D_{\chi_i})=D_{\chi^{(1)}_{ij}}$, where $\chi^{(1)}_{ij}$ is the character of $A$ associated to $\chi_i$ defined above.
\item Rank one reflexive sheaves $\mathcal{F}_{\chi_1},\cdots,
\mathcal{F}_{\chi_n}$ on $Z$ such that the linear equivalence \ref{fundamentalrel2} holds and for every $\gamma\in \N$, $\overline{\tau_j}^{\gamma}(\mathcal{F}_{\chi_i})=\mathcal{F}_{\chi^{(\gamma)}_{ij}}$,
where $\chi^{(\gamma)}_{ij}$ is defined above. Furthermore, $\overline{\tau_j}^{a_j}$ acts on the local sections of $\mathcal{F}_{\chi_i}$ as multiplication by $\exp(\frac{2\pi\sqrt{-1}k_{ij}}{m_i})$.
\end{enumerate}
\end{theorem}
\begin{proof}
(1) yields a flat abelian cover $q\colon Z\to Y$, where $Z=\spec(\oplus L_{\eta}^{-1})$. Next, define $X\coloneqq\spec(\oplus_{\chi \in N^*}\mathcal{F}_{\chi})$. The $\mathcal{O}_Z$-algebra structure is given by the morphisms $\mathcal{F}_{\chi}\otimes \mathcal{F}_{\chi^{\prime}}\to \mathcal{F}_{\chi\chi^{\prime}}$ which is uniquely determined by the restriction to $Z^0$ as $Z$ is normal. Let $g_k$ be the local equation for $D_k$, i.e., a function on $Z$
such that $D_k=\{g_k=0\}$. As $\mathcal{F}_{\chi_i}$ is locally-free on $Z^0$ (\cite{H80}, Prop 1.9), choose local generators $e_{\chi_i}$ and we set $e_{\chi^{(\gamma)}_{ij}}={\overline{\tau_j}^{\gamma}}^*(e_{\chi_i})$. The algebra structure on the restriction of $\oplus_{\chi \in N^*}\mathcal{F}_{\chi}$ on this local open subset is given by  
\[e_{\chi_i}e_{\chi_j}=e_{\chi_i\chi_j}\prod g_k^{\epsilon_{ij}^k}.\]
If we choose different generators $\tilde{e}_{\chi_i}$ satisfying the same equations $\tilde{e}_{\chi^{(\gamma)}_{ij}}={\overline{\tau_j}^{\gamma}}^*(\tilde{e}_{\chi_i})$, then the algebra structure is isomorphic to the above algebra. Due to relations $\overline{\tau_j}(D_{\chi_i})=D_{\chi^{(1)}_{ij}}$ one concludes that the morphism $\overline{\tau_j}^*$ defines a morphism of $\mathcal{O}_Z$-algebra on $\oplus_{\chi \in N^*}\mathcal{F}_{\chi}$. Finally, relations at the end of (3) ensure that the ${{\overline{\tau_j}^{a_j}}}^*$ satisfy the relation of \ref{presentation metabelian}. 
\end{proof}
With the notations of Theorem \ref{structure of metabelian}, we define $U_{\chi}\coloneqq q_*(\sF_{\chi})$ for every $\chi\in A^*$. This sheaves will be very useful in the sequel.
\begin{proposition} \label{Uisomorphism}
$\tau_j\colon X\to X$ induces an isomorphism $\tau_j^*\colon U_{\chi_i}\to	U_{\chi^{(1)}_{ij}}$ and $\tau_j^{*^{a_j}}\colon U_{\chi_i}\to U_{\chi_i}$ is $f\mapsto \exp(\frac{2\pi \sqrt{-1}k_{ij}}{m_i})f$.
\end{proposition} 
\begin{proof}
Let $V\subset Y$ be an open set. Then $U_{\chi_i}(V)=\mathcal{F}_{\chi_i}(q^{-1}(V))$ is the set of regular functions $ f\in \pi^{-1}(\mathcal{O}_X(V))$ such that $\sigma_j^*f=\exp(\frac{2\pi \sqrt{-1}\delta_{ij}}{m_i})f$. The relations in the presentation \ref{presentation metabelian} imply that $f\mapsto \tau^*f$ induces the morphism $\tau^*\colon U_i\to U_{\chi^{(1)}_{ij}}$ of $\mathcal{O}_Y$-modules. That $\tau_j^*$ and the last claim both follow also from the last equalities in \ref{presentation metabelian}, namely
$\tau_j^{a_j}=\sigma^{k_{1j}}_1\cdots \sigma^{k_{sj}}_s$. 
\end{proof}
Using Proposition \ref{Uisomorphism}, we can analyze the local behavior of the intermediate branch divisor $D_p$ and the sheaves $U_{\chi}$ as follows: Since $\overline{\tau}_j$ induces a bijection $A^*\to A^*$ by $\chi_i\mapsto \chi^{(1)}_{ij}$ it follows that it also induces $\overline{\tau}_j^*D_{p^0}=D_{p^0}$. So $D_{p^0}=\sum\limits_{\alpha}{(D_{p^0})}_{\alpha}$, which we view as a Weil divisor on $Z^0$, is invariant under the action of $\overline{\tau}_j:Z^0\to Z^0$ induced by $\tau_j$ for every $j$ and hence there exists an effective Cartier divisor $\Delta^j_{p^0}$ such that $(\tau_j^0)^*\Delta_{p^0}^j=D_{p^0}$, where $\tau_j^0=\tau_j|_{Z^0}$. Let $\Delta_{p^0}^{j,\alpha}=(\tau_j^0)^*{(D_{p^0})}_{\alpha}$. We define
\[\Delta_{p}^j=\overline{\Delta_{p^0}^j}, \hspace{0.5cm} \Delta_{p}^{j,\alpha}=\overline{\Delta_{p^0}^{j,\alpha}}.\]
Let $\mu_{\chi, \chi^{\prime}}:U_{\chi}\otimes U_{\chi^{\prime}}\to U_{\chi\chi^{\prime}}$ be the multiplication map. Notice that $U_1=q_*\sO_Z$ so in particular, 
\[\mu_{\chi, \chi^{-1}}:U_{\chi}\otimes U_{\chi^{-1}}\to U_1= \mathcal{O}_Y\oplus L_1\oplus \cdots \oplus L_{t-1},\]
where the $L_i$ are line bundles related to the abelian cover $q$ as in Theorem \ref{structure of metabelian}. Let us consider the multiplication map $\mu_{\chi_i,{(\chi^{(1)}_{ij})}^{-1}}$. By using the isomorphism in Proposition \ref{Uisomorphism}, we obtain a map $\mu_{\chi_i,{(\chi^{(1)}_{ij})}^{-1}}:U_{\chi_{i}}\otimes U_{\chi_i^{-1}}\to U_1=q_*\mathcal{O}_Z=\mathcal{O}_Y\oplus L_1\oplus \cdots \oplus L_{t-1}$. Write $\mu^{\alpha}_{\chi_i,{(\chi^{(1)}_{ij})}^{-1}}=\pr^{\alpha}\circ \mu_{\chi_i,{(\chi^{(1)}_{ij})}^{-1}}:U_{\chi_{i}}\otimes U_{\chi_i^{-1}}\to \sL_{\alpha}$ for $0\leq \alpha \leq t-1.$ We may therefore consider $\mu^{\alpha}_{\chi_i,{(\chi^{(1)}_{ij})}^{-1}}\in \h^0(Y, U_{\chi_{i}}\otimes U_{\chi_i^{-1}}\otimes L^{-1}_{\alpha})$. For each $y\in Y$, let $\mu^{\alpha}_{\chi_i,{(\chi^{(1)}_{ij})}^{-1},y}$ be the stalk of $\mu^{\alpha}_{\chi_i,{(\chi^{(1)}_{ij})}^{-1}}$ at $y$. Note that since $Z$ is normal, the singular locus is of $\codim\geq 2$, so the divisor of zeros of $\mu^{\alpha}_{\chi_i,{(\chi^{(1)}_{ij})}^{-1}}$ is determined by restricting to $Z^0$ and $Y^0$. So we may (and do) assume that $Z$ is smooth. If $u_{\chi_i}, u_{\chi_i^{-1}}$ are basis elements of $U_{\chi_i}$ and $U_{\chi_i^{-1}}$ repectively, then $u_{\chi_i}\otimes u_{\chi_i^{-1}}$ is a basis for $U_{\chi_{i}}\otimes U_{\chi_i^{-1}}$. Note that $U_{\chi_i}$ and $U_{\chi_i^{-1}}$ are considered as regular functions on a neighborhood of $\pi^{-1}(y)$ such that $\sigma_r^*(u_{\chi_i})=\chi_i(\sigma_r)u_{\chi_i}$ (analogous for $u_{\chi_i^{-1}}$). Choose local analytic coordinates such that $Z$ is given locally by the equation $z^t=y_1$. Suppose $e_{\chi_i}$ is basis of ${\sF}_{\chi_i}$ and $e_{\chi^{-1}_i}$ is basis of $\sF_{\chi^{-1}_i}$. Let 
\[u_{\chi_i}=1\cdotp e_{\chi_i}, u_{\chi_i^{-1}}=z\cdotp e_{\chi_i^{-1}}\]
Notice that we have $\tau_j^*(e_{\chi_i^{-1}})=e_{{(\chi^{(1)}_{ij})}^{-1}}$ and $\tau_j^*(z)=\zeta_{a_j}\cdotp z$. Replacing these relations gives \[\mu^{\alpha}_{\chi_i,{(\chi^{(1)}_{ij})}^{-1}}(u_{\chi_i}\otimes u_{\chi_i^{-1}})=e_{\chi_i}\cdotp {\tau_j^*}(e_{{(\chi^{(1)}_{ij})}^{-1}})= \zeta_t\cdotp {(e_{\chi_i}e_{\chi_i^{-1}})}_{\alpha}=zb_{p,\alpha},\]
where $b_{p,\alpha}$ is a local equation of $(D_{p^0})_{\alpha}$ and the last equality is due to \ref{fundamentalrel2}. 
\subsection{Invariants of metabelian covers}
Consider the metabelian cover $\pi \colon X\xrightarrow{p}Z\xrightarrow{q} Y$ with $p,q$ abelian covers of degree $m$ and $t$ respectively. Therefore, using Theorem \ref{singular canonical}, one can compute the invariants of the cover. Indeed, let $\sL^{\prime}_q$ be the reflexive sheaf associated to the abelian cover $q$ as in Theorem \ref{singular canonical} and let $\sF_i=\sF_{\chi_i}$ be the reflexive sheaves in the structure theorem \ref{structure of metabelian}. We have 
\begin{equation}
\begin{aligned}
\pi_*\omega_X=q_*(p_*\omega_X)=q_*(\omega_Z\otimes p_*\mathcal{O}_X) =q_*(q^*(\omega_Y\otimes {\sL^{\prime}}^{t-1}_q)\otimes (\oplus \mathcal{F}_i))= \\
q_*((\oplus \mathcal{F}_i)\otimes q^*(\omega_Y\otimes {\sL^{\prime}}^{t-1}_q))=q_*(\oplus \mathcal{F}_i)\otimes \omega_Y\otimes {\sL^{\prime}}^{t-1}_q=\\
(\oplus U_i)\otimes \omega_Y\otimes {\sL^{\prime}}^{t-1}_q=\oplus (\omega_Y\otimes {\sL^{\prime}}^{t-1}_q\otimes U_i))
\end{aligned}
\end{equation}
In the above, $U_j=q_*(\sF_j)$ and we have used the projection formula together with the fact that $\omega_Z=q^*(\omega_Y\otimes {\sL^{\prime}}^{t-1}_q)$ by Theorem \ref{singular canonical}.
\section*{Results for metacyclic covers} 
As mentioned in Definition \ref{defmetabelian}, A finite  metacyclic group $G$ is an extension of a finite cyclic group by a finite cyclic group, namely the following special case of \ref{metabelian extension}:
\begin{equation} \label{metacyclic extension}
0\to\langle\sigma\rangle\to G\to\langle\tau\rangle\to 0
\end{equation}
As a special case of metabelian groups, metacyclic groups are precisely the groups with the following presentation which also follows directly from presentation \ref{metabelian extension}.
\begin{equation} \label{metacyclic relations}
G=G_{m,k,t,r}\coloneqq \langle \sigma,\tau| \sigma^m=1, \sigma^k=\tau^t, \sigma \tau=\tau\sigma^r\rangle.
\end{equation} 
The numbers $m,k,t\text{ and }r$ are subject to the following conditions
\begin{equation} \label{metacyclic conditions}
r^t\equiv 1 \text{ (mod }m), kr\equiv k \text{ (mod }m) , \text{ and } \gcd(r,m)=1.
\end{equation}
Note that it follows from the presentation \ref{metacyclic relations} that $|G|=mt$. Let us remark that in the presentation \ref{metacyclic relations}, we always mean that $\ord(\sigma)=m$ and $\ord(\overline{\tau})=t$, where $\overline{\tau}$ is the class of $\tau$ in $G/\langle \sigma\rangle$. This is equivalent to saying that $t$ is the smallest positive integer such that $\tau^t\in \langle \sigma\rangle$.  This assumption together with the relations and conditions in \ref{metacyclic relations} and \ref{metacyclic conditions} imply that $\ord(\tau)=\frac{mt}{\gcd(k,m)}$. In particular, $G$ is an abelian (in fact cyclic) group if $\gcd(k,m)=1$.  Since we are mainly interested in non-abelian covers, we may and do henceforth assume that $\gcd(k,m)>1$. Furthermore, we emphasize that the quadruple $(m,k,t,r)$ is \emph{not} an invariant of the group, i.e., it may happen that $G_{m,k,t,r}\cong G_{m^{\prime},k^{\prime},t^{\prime},r^{\prime}}$, but $\{m,k,t,r\}\neq \{m^{\prime},k^{\prime},t^{\prime},r^{\prime}\}$. For instance $G_{36,6,0,19}\cong G_{12,18,0,7}$.
If $t=\ord(\tau)$, then we take $k=m$ and the group $G_{m,k,t,r}$ is called \emph{split}. In this case, $G=AN$ where $A=\langle\sigma\rangle$ and $N=\langle\overline{\tau}\rangle$. \par
Before exploring the metacyclic covers in some detail, let us mention that the structure theorem for metabelian covers, Theorem \ref{structure of metabelian}, takes the followsing special form when the Galois group is metacyclic.
\begin{theorem} (Structure theorem for metacyclic covers) \label{structure of metacyclic}
Let $Y$ be a smooth algebraic variety. The following data determine a $G_{m,k,t,r}$-cover $\pi \colon X\to Y$.
\begin{enumerate}
\item A line bundle $\mathcal{L}$ and an effective reduced divisor $B_q$ such that $\mathcal{L}^{\otimes t}=\mathcal{O}_Y(-B_q)$
\item Reduced effective Weil divisors $D_1\cdots, D_m$ on $Z\coloneqq\spec(\mathcal{O}_Y\oplus \mathcal{L}^{-1}\oplus \cdots \oplus \mathcal{L}^{-(t-1)})$ such that $\overline{\tau}(D_{i})=D_{\overline{ri}}$. Here $\overline{\tau}$ is an automorphism of order $t$ of $q\colon Z\to Y$. 
\item Rank one reflexive sheaves $\mathcal{F}_1,\cdots, \mathcal{F}_m$ on $Z$ flat over $\mathcal{O}_Y$ such that the following relations are satisfied: $\overline{\tau}(\mathcal{F}_{i})=\mathcal{F}_{\overline{ri}}$ and such that $\overline{\tau}^{t}$ acts on the local sections of $\mathcal{F}_{i}$ as multiplication by $\exp(\frac{2\pi \sqrt{-1}k}{m})$.
\end{enumerate}
\end{theorem}
\begin{proof}
A special case of Theorem \ref{structure of metabelian}. However, we remark that in this case, $X\coloneqq\spec(\mathcal{O}_Z\oplus \mathcal{F}_1\oplus \cdots \oplus \mathcal{F}_{m-1})$ and the $\mathcal{O}_Z$-algebra structure is given by the morphisms $\mathcal{F}_i\otimes \mathcal{F}_j\to \mathcal{F}_{\overline{i+j}}$ (Here ${\overline{i+j}}$ is the sum in $\Z/m\Z$) which is uniquely determined by the restriction to $Z^0$ as $Z$ is normal. As $\mathcal{F}_i$ is locally-free on $Z^0$, there are local generators	$e_i$ and we set $e_{\overline{ri}}=\overline{\tau}^*(e_i)$. The algebra structure on the restriction of $\mathcal{O}_Z\oplus \mathcal{F}_1\oplus \cdots \oplus \mathcal{F}_{m-1}$ on this local open subset is given by 
\[e_ie_j=e_{\overline{i+j}}\prod g_s^{\epsilon_{ij}^s},\]
where $g_s$ is the local equation for $D_s$. 
If we choose different generators $\tilde{e_i}$ satisfying the same equations $\tilde{e}_{\overline{ri}}=\overline{\tau}^*(\tilde{e_i})$, then the algebra structure is isomorphic to the above algebra. Due to relations $\overline{\tau}(D_{i})=D_{\overline{ri}}$ one concludes that the morphism $\overline{\tau}^*$ defines a morphism of $\mathcal{O}_Z$-algebras on $\mathcal{O}_Z\oplus \mathcal{F}_1\oplus \cdots \oplus \mathcal{F}_{m-1}$. Finally, the proposed relations at the end of (3) ensure that this defines a morphism which satisfies the generator-relation \ref{metacyclic extension}, namely that $\overline{\tau}^*={\sigma^k}^*$.  	
\end{proof}
\subsection{Complex representations of metacyclic groups} \label{repres of metacyclic}
Complex irreducible representations of finite  metacyclic groups have been determined in \cite{B69}. Let us explain, albeit with slightly different notations, the results of \cite{B69}. Let $U_m=\{z\in\mathbb{C}\mid z^m=1\}$ and consider the map
$\alpha_r\colon U_m\to U_m, \zeta\mapsto\zeta^r$. So the group $\langle\alpha_r \rangle$ acts on $U_m$. For $\zeta\in U_m$, let $t(\zeta)$ be the size of the orbit of $\zeta$ under the action of $\langle\alpha_r \rangle$. Then $t(\zeta)|t$. Let $\{\zeta_1,\cdots, \zeta_s\}$ be a set of representatives of the distinct orbits of $U_m$. Denote the size of an orbit with representaitve $\zeta_i$ by $t(\zeta_i)$ for $i=1,\cdots, s$. Complex irreducible representations are classified as follows. Consider the matrices $T_{\zeta,\theta}\in \gl(t(\zeta), \C)$ given by
\begin{equation} \label{irr.rep metacyclic}
T_{\zeta,\theta}(\sigma)=\begin{bmatrix}
\zeta & & &0\\
 & \zeta^r& &\\
& & \ddots &\\
0& & & \zeta^{r^{t(\zeta)-1}}
\end{bmatrix}, \hspace{3cm} T_{\zeta,\theta}(\tau)=
\begin{bmatrix}
0 & 0 & &  \dots & \theta \\
1 & 0 &  &\dots & 0 \\
0 & 1 & 0&\dots & 0 \\
\vdots &  &\vdots & \ddots & \vdots \\
0 & 0 & \dots & 1 & 0
\end{bmatrix}
\end{equation}
The (inequivalent) complex irreducible representations of $G_{m,k,t,r}$ are precisely $T_{\zeta_i,\theta}\in \gl(t(\zeta_i), \C)$ where $1\leq i\leq s$ and $\theta$ runs over the solutions of $\theta^{\frac{t}{t(\zeta_i)}}=\zeta_i^k$. Therefore the number of (inequivalent) complex irreducible representations of $G_{m,k,t,r}$ is equal to $\nu=\displaystyle t\sum_{i=1}^s\frac{1}{t(\zeta_i)}$. Let $\zeta_i=\xi_m^{l_i}$, where $\xi_m$ is a primitive $m$-th root of unity. We have 
\begin{theorem} \label{metacyclic decomposition}
Let $Y$ be a smooth variety and $\pi\colon X\to Y$ a flat $G_{m,k,t,r}$-cover with $X$ normal and let $p\colon X\to Z$ and $q\colon Z\to Y$ be the intermediate coverings of degrees $m$ and $t$ respectively. Then
\begin{equation}
\pi_*\mathcal{O}_X=\displaystyle\oplus_{i=1}^{\nu} (\pi_*\mathcal{O}_X)_{i},
\end{equation}
where the summands are as follows: If $\zeta_i$ is a representative of an orbit such that $t(\zeta_i)=t$, then 
\begin{equation}
(\pi_*\mathcal{O}_X)_{i}=U_{l_i}\oplus U_{r{l_i}}\oplus\cdots\oplus U_{r^{(t(\zeta_i)-1)}{l_i}},
\end{equation} 
is precisely the eigensheaf associated to the irreducible representation $T_{\zeta_i,\zeta_i^k}$ of \ref{irr.rep metacyclic}, where $U_j=q_*(\mathcal{F}_{\overline{j}}).$ If $t(\zeta_i)<t$, then $(\pi_*\mathcal{O}_X)_{i}$ is the sum of eigensapces comming from metacyclic intermediate covers $X/H\to Y$. 
\end{theorem}
\begin{proof}
If $t(\zeta_i)=t$, then it is clear from the description of the irreducible representations of metacyclic groups \ref{irr.rep metacyclic} that the above sum is in the eigenspace associated with the irreducible representation $\rho_i$. On the other hand, such an eigensheaf is a free $\sO_Y[G_{m,k,t,r}]$-module of rank $(\dim \rho_i)^2$. Since $t(\zeta_i)=t$, the above sum is precisely of rank $t^2=t(\zeta_i)^2=(\dim \rho_i)^2$. Consequently, $(\pi_*\mathcal{O}_X)_{i}$ is precisely the eigensheaf associated to $\rho_i$. If $t(\zeta_i)<t$, then the sections of the line bundles associated to the irreducible representation described in \ref{irr.rep metacyclic}, are invariant under a non-trivial subgroup $H$, hence they descend to regular functions on $X/H$. As any subgroup of a metacyclic group is itself metacyclic, the cover $X/H\to Y$ mentioned above is again a metacyclic cover with Galois group $H$. 
\end{proof}

\begin{example} (Catanese-Perroni \cite{CP}, p.23)
Let $\pi\colon X\to Y$ be a $D_n$-cover. Then if $n$ is odd, there is only one orbit $\{1\}$ in $U_n$ such that $t(1)=1$ and there are $\frac{n-1}{2}$ orbits with $t(\zeta_i)=2$. The sections arising from the bundle associated to this orbit are invariant under the cyclic subgroup $H=\langle \sigma\rangle$, so they descend to sections on $Z$ hence $\pi_*\mathcal{O}_X=\sO_Y\oplus\sL\bigoplus\limits_{i=1}^{\frac{n-1}{2}}(\pi_*\mathcal{O}_X)_i$, where $\sO_Y\oplus\sL=q_*\sO_Z$. If $n$ is even, there are two orbits $\{1\}, \{-1\}$ such that $t(1)=t(-1)=1$ and there are $\frac{n-2}{2}$ orbits with $t(\zeta_i)=2$. The sections arising from the bundle associated to the above mentioned two orbits are invariant under the subgroup $H=\langle\sigma^2,\tau\rangle$, so they descend to sections of the $\Z_2\times \Z_2$-cover $X/H\to Y$. Consequently we have $\pi_*\mathcal{O}_X=\sO_Y\oplus\sL\oplus\sM\oplus\sN\bigoplus\limits_{i=1}^{\frac{n-2}{2}}(\pi_*\mathcal{O}_X)_i$, where the first summands correspond to the intermediate abelian cover $X/H$. 
\end{example}
If $t$ is a prime number, Theorem \ref{metacyclic decomposition} gives more information about $G_{m,k,t,r}$-covers.
\begin{corollary} \label{t is prime}
Let $Y$ be a smooth complex algebraic variety and $\pi\colon X\to Y$ a flat $G_{m,k,t,r}$-cover with $X$ normal and suppose that $p\colon X\to Z$ and $q\colon Z\to Y$ are the intermediate coverings of degrees $m$ and $t$ respectively, where $t$ is a prime number. Then
\begin{equation}
\pi_*\mathcal{O}_X=\displaystyle\oplus_{i=1}^{\nu} (\pi_*\mathcal{O}_X)_{i}.
\end{equation}
The eigensheaves $(\pi_*\mathcal{O}_X)_{i}$ are described as follows: Let $b=\gcd(r-1,m)$ and $h=\frac{m-b}{t}$ (this is an integer).  If $i\geq h$, then 
\begin{equation}
(\pi_*\mathcal{O}_X)_{i}=U_{l_i}\oplus U_{r{l_i}}\oplus\cdots\oplus U_{r^{(t(\zeta_i)-1)}{l_i}},
\end{equation} 
is the eigensheaf associated to the irreducible representation $T_{\zeta_i,\zeta_i^k}$ of \ref{irr.rep metacyclic}, where $U_j=q_*(\mathcal{F}_{\overline{j}}).$
\end{corollary}
\begin{proof}
By the description of the irrducible representations in \ref{irr.rep metacyclic}, an $\langle\alpha_r\rangle$-orbit with representative $\zeta_i=\xi_m^{l_i}$ satisfies $t(\zeta_i)=1$ if and only if $m|l_i(r-1)$. Therefore there are $b=\gcd(r-1,m)$ orbits with $t(\zeta_i)=1$. As $t$ is a prime number and $t(\zeta_i)|t$, all the other orbits have $t(\zeta_i)=t$ and there are $h=\frac{m-b}{t}$ of these. Theorem \ref{metacyclic decomposition} then yields the remaining assertions.  
\end{proof}
\begin{proposition}
\item The automorphism $\tau\colon X\to X$ induces an isomorphism $\tau^*\colon U_i\to U_{\overline{ri}}$ and $\tau^{*^t}\colon U_i\to U_i$ is the map $f\mapsto \exp(\frac{2\pi \sqrt{-1}k}{m})f$. 
\end{proposition}
\begin{proof}
Let $V\subset Y$ be an open set. Then $U_i(V)=\mathcal{F}_i(q^{-1}(V))$ is the set of regular functions $ f\in \pi^{-1}(\mathcal{O}_X(V))$ such that $\sigma^*f=\exp(\frac{2\pi \sqrt{-1}i}{m})f$. The relation $\sigma\tau=\tau\sigma^r$ implies that $f\mapsto \tau^*f$ induces a morphism $\tau^*\colon U_i\to U_{\overline{ri}}$ of $\mathcal{O}_Y$-modules. The last claim follows since $r^t\equiv 1$ (mod $m$) and $\tau^t=\sigma^k$. 
\end{proof}
\subsection{Function fields of metacyclic covers}
Consider a metacyclic Galois covering $f\colon X\to Y$ with the Galois group $G_{m,k,t,r}$ defined as above with the intermediate cyclic coverings $p\colon X\to Z$ and $q\colon X\to Z$ of degrees $m$ and $t$ respectively. Recall from Sction~\ref{Kummer}, these morphisms give rise to field extension $\C(Y)\subset \C(Z)\subset \C(X)$ of the corresponding function fields. Let $A=\langle\sigma\rangle$ and $N=\langle\overline{\tau}\rangle$ with $\text{ord}(\sigma)=m$ and $\text{ord}(\overline{\tau})=t$. Then
\begin{gather}
\C(X)=\C(Z)(\sqrt[m]{g})=\C(Z)(v), g\in \C(Z), \sigma\cdotp v=\zeta_{m}\cdotp v,\\
\C(Z)=\C(Y)(\sqrt[t]{f})=\C(Y)(w), f\in \C(Y), \overline{\tau}.w=\zeta_{t}\cdotp w
\end{gather}
where $\zeta_{m}$ (resp. $\zeta_{t}$) denotes a primitve $m$-th (resp. $t$-th) root of unity. In particular, $\C(X)=\C(Y)(v,w)$. In this subsection we describe the fucntion fields and the extensions above following \cite{CP}, \cite{P1} and \cite{Tok02}.  
\begin{proposition} \label{fieldequations}
Let $\pi\colon X\to Y$ be a metacyclic Galois covering with the Galois group $G_{m,k,t,r}$ defined as above with the intermediate cyclic coverings $p\colon X\to Z$ and $q\colon Z\to Y$. Then the following hold:
\begin{enumerate}
\item There exists $\alpha\in \C(Z)$, such that $\tau(x)=\alpha x^r$.
\item Let $c=\frac{r^t-1}{m}$. Then $g^c=\frac{\xi_m^k}{\prod\limits_{i=1}^t \overline{\tau}^{t-i}(\alpha)\alpha^{r^{t-1}}}$. 
\end{enumerate}
\end{proposition}
\begin{proof}
\begin{enumerate}
\item Consider the element $x^{-r}\tau(x)$. It holds that: 
\[\sigma(x^{-r}\tau(x))=\sigma(x^{-r})\sigma\tau(x)=\sigma(x^{-r})\tau\sigma^r(x)=\xi_m^{-r} x^{-r}\tau(\xi_m^r x)=x^{-r}\tau(x).\]
It follows that $x^{-r}\tau(x)$ is invariant under $\sigma$, i.e., $x^{-r}\tau(x)\in \C(X)^{\langle\sigma \rangle}=\C(Z)$. This proves the first claim.  
\item Using the formula in (1) with $\alpha\in \C(Z)$, one calculates that 
\[\xi_m^kx=\sigma^k(x)=\tau^t(x)=\prod_{i=1}^t\tau^{t-i}(\alpha)\alpha^{r^{t-1}}x^{r^t}=\prod_{i=1}^t \overline{\tau}^{t-i}(\alpha)\alpha^{r^{t-1}}g^c x.\] 
Which settles (2). 
\end{enumerate}	
\end{proof}
\begin{lemma}
\label{tau powers}
\begin{enumerate}
\item If $h\in \mathbb{C}(Z)$, then
$F\coloneqq h\tau(h)\tau^2(h)\cdots\tau^{t-1}(h)\in \mathbb{C}(Y)$.
\item Let $P\coloneqq x\tau(x)\tau^2(x)\cdots\tau^{t-1}(x)$. Then there exists $u\in \N$ with $u|\gcd(r-1,m)$ such that
$P^{u}\in \mathbb{C}(Y)$.
\end{enumerate}
\end{lemma}
\begin{proof}
\begin{enumerate}
\item Since $h\in \mathbb{C}(Z), F=h\tau(h)\tau^2(h)\cdots\tau^{t-1}(h)=h\overline{\tau}(h)\overline{\tau}^2(h)\cdots\overline{\tau}^{t-1}(h)\in
\mathbb{C}(Z)$. It therefore suffices to show that this element is
invariant under the action of $\tau$. Indeed
\begin{equation}
\label{mult of powers}
\tau(F)=\tau(h)\tau^2(h)\cdots\tau^{t-1}(h)\tau^{t}(h).
\end{equation}
But $\tau^{t}=\sigma^k$ and so $\tau^{t}(h)=\sigma^k(h)=h$,
because $h\in \mathbb{C}(Z)$ by assumption and hence it is
invariant under $\sigma$.
\item By Proposition \ref{fieldequations}, $\alpha=x^{-r}\tau(x)\in\mathbb{C}(Z)$ so using \ref{mult of powers} above, we have
\[\alpha\tau(\alpha)\tau^2(\alpha)\cdots\tau^{t-1}(\alpha)=\xi_m^k(x\tau(x)\tau^2(x)\cdots\tau^{t-1}(x))^{-(r-1)}\in\mathbb{C}(Y).\]
Note that $\tau^{t}(x)=\sigma^k(x)=\xi_m^kx$. For the last
assertion, notice that $x^m=g\in \mathbb{C}(Z)$ and it follows by
setting $h=x^m$ in \ref{mult of powers} that
\[h\tau(h)\tau^2(h)\cdots\tau^{t-1}(h)=(x\tau(x)\tau^2(x)\cdots\tau^{t-1}(x))^{m}\in\mathbb{C}(Y).\]
Now take $u$ to be the smallest positive integer such that $(x\tau(x)\tau^2(x)\cdots\tau^{t-1}(x))^{u}\in\mathbb{C}(Y)$. By the minimality of $u$, it follows that $u|\gcd(r-1,m)$.
\end{enumerate}
\end{proof}
In the special case $t=2$, the above results can help to determine
the structure of metacyclic extensions.
\begin{proposition}
Let $\C(Y)\subset \C(X)$ be a $G_{m,k,2,r}$- extension. There exist $a\in \C(Y)$ and $x,P\in \C(X)$ with
\[x^{2m}-2ax^m+P^m=0,\]
where $P\in \C(X)$ is such that there exists $u|\gcd(r-1,m)$ with $P^u\in \C(Y)$ (in particular, $P^m\in \C(Y)$) and $\C(X)=\C(Y)(x)$ and the $G_{m,k,2,r}$-action is given by $\sigma(x)=\xi_{m}x$ where $\xi_{m}$ is a primitive $m$-th root of unity in $\C$ and $\tau(x)=\frac{P}{x}$.\par 
Conversely, given $m,k,r$ satisfying \ref{metacyclic conditions} (with $t=2$) and $a\in \C(Y)$, if $P=\alpha x^{r+1}$ is such that $\alpha\in \C(Y)[x^m]$ and $x^{2m}-2ax^m+P^m$ is irreducible (in $\C(Y)[x]$) then $\frac{\C(Y)[x]}{(x^{2m}-2ax^m+P^m)}$ is a $G_{m,k,2,r}$- Galois extension of $\C(Y)$ with the $G_{m,k,2,r}$- action given by $\sigma(x)=\xi_m x$ and $\tau(x)=\frac{P}{x}=\alpha x^{r}$.
\end{proposition}
\begin{proof}
Let $P\coloneqq x\tau(x)\in \C(X)$ and $u\in\N$ be the smallest positive integer such that $P^u\in\C(Y)$ as in Lemma \ref{tau powers}. By Lemma \ref{tau powers} (2), $u|\gcd(r-1,m)$. In particular $P^m\in\C(Y)$. More explicitely, since $t=2$, we may write $x^m=g=a+z$ with $a\in \C(Y)$ and we have
\[P^m=(x\tau(x))^{m}=x^{m}\overline{\tau}(x^{m})=(a+z)(a-z)=a^2-z^2=a^2-f.\]
So $f=a^2-P^m$. On the other hand, $x^{m}=a+z$ and the above calculation gives
\[x^{2m}-2ax^m+P^m=0.\]
Of course, by Proposition \ref{fieldequations}, we know that $P=\alpha x^{r+1}$ with $\alpha \in \C(Z)$.  

For the converse statement, one checks that by the choice of $P$, there is an action of $G_{m,k,2,r}$ on $\C(X)=\frac{\C(Y)[x]}{(x^{2m}-2ax^m+P^m)}$. This action is Galois as the latter contains all of the conjugates of $x$ under $G_{m,k,2,r}$.  
\end{proof}
\subsection*{Special case: dicyclic Covers}
The so-called dicyclic group is in our compact notation the metacyclic group $G_{2n,n,2,-1}$, where $n\in \Z$. They are therefore precisely the extensions of $\Z_2$ by cyclic groups. 
In this case as in the proof of ~\ref{fieldequations} (1), we see that since $t=2$, the element $x\tau(x)\in \C(X)^{\langle\sigma \rangle}=\C(Z)$
and hence $x\tau(x)=a+bz$ with $a,b\in \C(Y)$. Futhermore, we have $\overline{\tau}(x\tau(x))=-x\tau(x)$ which forces $a=0$, and we may assume $b=1$. That is
$x\tau(x)=z$ or $\tau(x)=\frac{z}{x}$. In particular, $\tau(x)^2=\frac{f}{x^2}$ with $z^2=f\in \C(Y)$. 
Now, $x^{2n}=g\in \C(Z)$ and hence $x^{2n}=g=c+dz$ with $c,d\in \C(Y)$. 
\[f^n=z^{2n}=x^{2n}\tau(x^{2n})=x^{2n}\overline{\tau}(x^{2n})=x^{2n}(c-dz)=(c+dz)(c-dz)=c^2-d^2f\]
On the other hand, $x^{2n}=g=c+dz$ together with the above computations imply 
\[x^{4n}-2cx^{2n}+c^2-d^2f=x^{4n}-2cx^{2n}+f^n=0,\]
with $c,d,f\in \C(Y)$. We summarize this discussion in the following 
\begin{corollary}
Let $\C(Y)\subset \C(X)$ be a $G_{2n,n,2,-1}$- extension. Then there exists $c,f\in \C(Y)$ and $x\in \C(X)$ such that $\C(X)=\C(Y)(x)$ and
\[x^{4n}-2cx^{2n}+c^2-d^2f=x^{4n}-2cx^{2n}+f^n=0,\]
where $\sigma(x)=\xi_{2n}x$ and $\tau(x)^2=\frac{f}{x^2}$ and $\xi_{2n}$ is a primitive $2n$-th root of unity in $\C$. Conversely, given $c,f\in \C(Y)$ and $x\in \C(X)$ such that $x^{4n}-2cx^{2n}+f^n$ is irreducible in $\C(Y)[x]$, then $\frac{\C(Y)[x]}{(x^{4n}-2cx^{2n}+f^n)}$ is a $G_{2n,n,2,-1}$- extension of $\C(Y)$ with the group action described above. 
\end{corollary}


\begin{thebibliography}{00}
\bibitem{B69}
B. G. Basmaji, \emph{Complex representations of metacyclic groups.} The American Mathematical Monthly, 86 (1): 47-48.
\bibitem{BHPV}
Barth, Wolf P., Hulek, Klaus, Peters, Chris A.M., Van de Ven, Antonius, \emph{Compact Complex Surfaces.}(2004) Ergebnisse der Mathematik und ihrer Grenzgebiete. 3. Folge., 4, Springer-Verlag, Berlin. 
\bibitem{Com30}
Comessatti, A, \emph{Sulle superficie multiple cicliche.} Rendiconti Seminario
Padova 1, 1–45 (1930).
\bibitem{CP}
F. Catanese, F. Perroni, \emph{Dihedral Galois covers of algebraic varieties and the simple cases.} To appear in the Journal of Geometry and Physics. 
\bibitem{C10}
F. Catanese, \emph{Irreducibility of the space of cyclic covers of algebraic curves of fixed numerical type and the irreducible components of Sing$(\overline{\textswab{M}_g})$.} Advances in geometric analysis, 281–306, Adv. Lect. Math. (ALM), 21, Int. Press, Somerville, MA, 2012.
\bibitem{CLP}
F. Catanese, M. L\"onne, F. Perroni, \emph{Irreducibility of the space of dihedral covers of the projective line of a given numerical type.} Rend. Lincei Mat. Appl.22 (2011), 291-309.
\bibitem{Ea11}
Easton, R, \emph{$S_3$-covers of Schemes.} Canadian Journal of Mathematics, 63(5), (2011), 1058-1082. doi:10.4153/CJM-2011-045-8
\bibitem{H80}
R.  Hartshorne, \emph{Stable  reflexive  sheaves.} Math. Ann. 254(1980),  no.  2,  121–176. MR597077 (82b:14011)
\bibitem{MZ}
A. Mohajer, K. Zuo, \emph{On Shimura subvarieties generated by families of abelian covers of $\mathbb{P}^{1}$.} Journal of Pure and Applied Algebra, 222 (4), 2018, 931-949. 
\bibitem{M}
A. Mohajer, \emph{On the Prym map of Galois coverings.} In preparation. 
\bibitem{P1}
R. Pardini, \emph{Abelian covers of algebraic varieties.} J.reine angew. Math., 417 (1991), 191-213.
\bibitem{Per81}
U. Persson, \emph{Chern invariants  of  surfaces of  general type.}  Compo. Math.43(1981), 3–58.
\bibitem{Ri}
M. Ried, \emph{Canonical 3-folds}. Journ\'ees de Géometrie Alg\'ebrique d'Angers, Juillet 1979/Algebraic Geometry, Angers, 1979, pp. 273-310, Sijthoff \& Noordhoff, Alphen aan den Rijn–Germantown, Md., 1980.
\bibitem{Tok02}
H. Tokunaga, \emph{Galois covers for $D_4$ and $A_4$ and their applications.} Osaka J. Math., Volume 39, Number 3 (2002), 621-645.
\end{thebibliography}
\end{document}